\documentclass{amsart}
\usepackage{amsmath,amssymb,amsfonts}
\usepackage{graphics}
\usepackage{amsthm}
\usepackage{latexsym}
\usepackage{epsf}

\theoremstyle{plain}
\newtheorem{theorem}[equation]{Theorem}

\newtheorem{prop}[equation]{Proposition}

\theoremstyle{definition}
\newtheorem{definition}[equation]{Definition}

\newtheorem{remark}[equation]{Remark}

\numberwithin{equation}{section}

\input xy
\xyoption{all}

\begin{document}

\title[Action graphs]{Action graphs, rooted planar forests, and self-convolutions of the Catalan numbers}

\author{Julia E.\ Bergner, Cedric Harper, Ryan Keller, and Mathilde Rosi-Marshall}
\address{Department of Mathematics, University of Virginia,
Charlottesville, VA 22904}
\email{jeb2md@virginia.edu} 
\email{cbh2ta@virginia.edu} 
\email{rjk7bb@virginia.edu} 
\email{mgr3km@virginia.edu}

\thanks{The first- and fourth-named authors were partially supported by NSF CAREER award DMS-1352298.  The third-named author was partially supported by the USOAR program at the University of Virginia in 2017-18.}

\maketitle

\begin{abstract}
We show that families of action graphs, with initial graphs which are linear of varying length, give rise to self-convolutions of the Catalan sequence.  We prove this result via a comparison with planar rooted forests with a fixed number of trees.
\end{abstract}

\section{Introduction}

In a recent paper \cite{alv}, the first-named author, Alvarez, and Lopez showed that a certain family of directed graphs, called \emph{action graphs}, give a new way to produce the Catalan numbers.  These graphs arose in work of the first-named author and Hackney, with the goal of understanding the structure of a rooted category action on another category \cite{reedy}.  However, in that work we considered a much larger collection of such graphs than those considered in the paper with Alvarez and Lopez.  In this paper, we take up the question of what sequences arise from the other families of action graphs, and we prove that they produce self-convolutions of the Catalan sequence.  

The Catalan sequence arises in many contexts in combinatorics  \cite[A000108]{oeis}; a long list of ways to obtain the Catalan numbers is found in Stanley's books \cite{stanleycat}, \cite{stanley} and online addendum \cite{catadd}; see also Koshy's book \cite{koshy}.  Recall that the 0th Catalan number is $C_0=1$, and, for any $n \geq 0$, the $(n+1)$st Catalan number $C_{n+1}$ is given by the formula
\[ C_{n+1} = \sum_{i=0}^n C_i C_{n-i}. \]

Here, we are interested in convolutions of this sequence with itself.  To this end, let us recall the definition of the convolution of sequences.

\begin{definition}
Given two sequences $A$ and $B$, their \emph{convolution} is given by
\[ (A\ast B)_n=\sum_{m=0}^n A_m B_{n-m} \]
The \emph{self-convolution} of a sequence $A$ is the convolution $A \ast A$, which we denote by $A^1$.  More generally, we let 
\[ A^k = \underbrace{A \ast \cdots \ast A}_{k+1}, \]
and in particular $A^0$ is simply the original sequence $A$.
\end{definition}

\begin{remark}
The notation we have chosen is arguably confusing, as one might prefer to have $A^1$ denote the sequence itself and $A^2$ the convolution $A \ast A$.  However, we want to think of the index $k$ as telling us how many convolutions we have done, rather than how many copies of the sequence have been convolved.  Furthermore, our indexing is convenient for the comparison we undertake here.

So as not to encumber the notation unnecessarily, we denote the $n$th term of the $k$th convolution by $A_n^k$ rather than $(A^k)_n$.
\end{remark}

Applying this definition to the Catalan sequence, observe that the sequence $C^1$ begins as
\[ C^1_0=1, C^1_1=2, C^1_2=5, C^1_3=14, \ldots, \]
which is simply a shift of the Catalan sequence, whereas $C^2$ begins to look quite different:
\[ C^2_0=1, C^2_1=3, C^2_2=9, \ldots. \]

In Section \ref{kext}, we define the general families of action graphs that we want to consider.  Then in Section \ref{prf} we introduce planar rooted forests and prove a connection with self-convolutions of Catalan numbers.  Finally, in Section \ref{comp} we give a bijection between certain sets of planar rooted forests and our general action graphs.

\section{$k$-Extended action graphs} \label{kext}

In this section, we recall the definition of action graph as it appeared in previous work \cite{alv} and generalize to the setting we wish to consider.  To begin, let us recall the definition of a directed graph that we use here.

\begin{definition}
A \emph{directed graph} is a pair $G=(V,E)$ where $V$ is a set whose elements are called \emph{vertices} and $E$ is a set of ordered pairs of vertices, called \emph{edges}.  Given an edge $e= (v, w)$, we call $v$ the \emph{source} of $e$, denoted by $v=s(e)$, and call $w$ the \emph{target} of $e$, denoted by $w=t(e)$.
\end{definition}

For the directed graphs that we consider here, we assume that, for every $v \in V$, we have $(v,v) \in E$.  While we could think of these edges as loops at each vertex, we prefer to regard them as \emph{trivial edges} given by the vertices.  Otherwise, we assume there are no loops or multiple edges, so there is no ambiguity in the definition as we have given it.

\begin{definition}
A \emph{(directed) path} in a directed graph is a sequence of edges $e_1, \ldots, e_n$ such that $t(e_i)=s(e_{i+1})$ for each $1\leq i<n$.  For paths consisting of more than one edge, we require all these edges to be nontrivial.  We call $s(e_1)$ the \emph{initial vertex} of the path and $t(e_k)$ the \emph{terminal vertex} of the path.
\end{definition}

We also equip directed graphs with a labeling by a set $S$, i.e., a specified function $V \rightarrow S$.  Here, $S$ is either the set of natural numbers $\mathbb N$, which we assume to include 0, or $\mathbb N \cup \{-k, -k+1, \ldots, -1\}$ for some $k \geq 1$.

We begin with the definition of action graph from \cite{alv}, using some additional terminology to distinguish it from later examples.

\begin{definition}
For each natural number $n$, the (\emph{0-extended}) \emph{action graph} $A^0_n$ is the directed graph labeled by $\mathbb N$ which is defined inductively as follows.  The action graph $A^0_0$ is defined to be the graph with one vertex labeled by 0 and no nontrivial edges.  Inductively, given $A^0_n$, define $A^0_{n+1}$ by freely adjoining new edges by the following rule.  For any vertex $v$ of $A^0_n$ labeled by $n$, consider all paths in $A^0_n$ with terminal vertex $v$.  For each such path, adjoin a new edge whose source is the initial vertex of that path, and whose target is a new vertex which is labeled by $n+1$.
\end{definition}

Thus, the first few 0-extended action graphs can be depicted as follows:
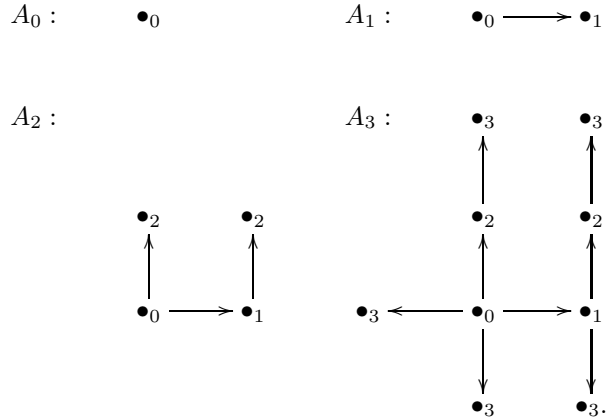
\[ \xymatrix{A_0: & \bullet_0 && A_1: & \bullet_0 \ar[r] & \bullet_1  \\
A_2: &&& A_3: & \bullet_3 & \bullet_3 \\
& \bullet_2 & \bullet_2 && \bullet_2 \ar[u] & \bullet_2 \ar[u] \\
& \bullet_0 \ar[u] \ar[r] & \bullet_1 \ar[u] & \bullet_3 & \bullet_0 \ar[u] \ar[d] \ar[l] \ar[r] & \bullet_1 \ar[u] \ar[d] \\
&&&& \bullet_3 & \bullet_3.} \]

The first main result of \cite{alv} is the following theorem.

\begin{theorem} \cite[2.4]{alv} \label{catalanthm}
The number of new vertices (and edges) adjoined to $A^0_n$ to obtain $A^0_{n+1}$ is given by the $(n+1)$st Catalan number $C_{n+1}$.
\end{theorem}

In this paper, we want to prove an analogous result for more general families of action graphs. Above, we defined a sequence $A^0$ of 0-extended action graphs; we now want to give more general starting graphs which, using the same inductive procedure, produce sequences $A^k$ for any $k \geq 1$. 

\begin{definition}
For any $k \geq 1$, let $A^k_0$ be the directed graph labeled by $\mathbb N \cup \{-k, \ldots, -1\}$ given by
\[ \xymatrix@1{\bullet_{-k} \ar[r] & \bullet_{-k+1} \ar[r] & \cdots \ar[r] & \bullet_{-1} \ar[r] & \bullet_0.} \]  Then for any $n \geq 1$ the directed graph $A^k_n$ is defined using the same inductive rule as before.  We call such an action graph $k$-\emph{extended}.
\end{definition}

The first few 1-extended action graphs can be depicted as follows:
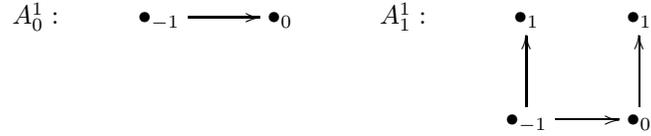
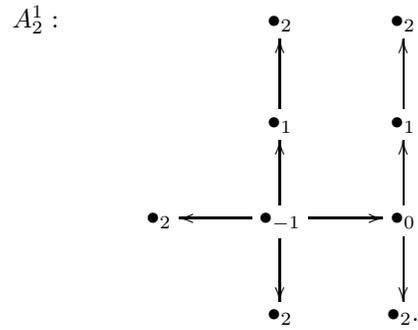
\[ \xymatrix{A_0^1: & \bullet_{-1} \ar[r] & \bullet_0 & A_1^1:& \bullet_1 & \bullet_1 \\
&&&& \bullet_{-1} \ar[u] \ar[r] & \bullet_0 \ar[u] \\
A_2^1: && \bullet_2 & \bullet_2 && \\
&& \bullet_1 \ar[u] & \bullet_1 \ar[u] && \\
& \bullet_2 & \bullet_{-1} \ar[u] \ar[d] \ar[l] \ar[r] & \bullet_0 \ar[u] \ar[d] && \\
&&\bullet_2 & \bullet_2. &&} \]
Observe that the graphs themselves look like the 0-extended action graphs, but the index of each graph is shifted by one and likewise there is a shift in the labeling.

We get more a more interesting family of 2-extended action graphs:
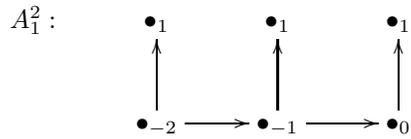
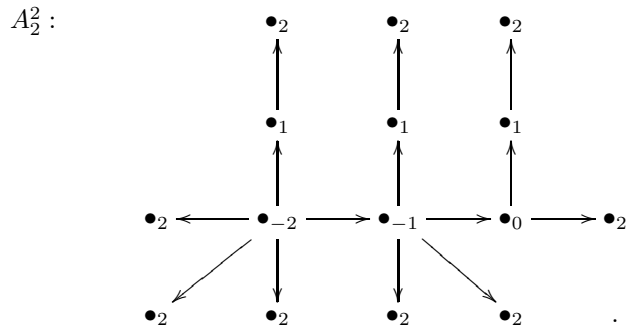
\begin{equation} \label{1-ext}
\xymatrix{A^2_0: & \bullet_{-2} \ar[r] & \bullet_{-1} \ar[r] & \bullet_0 & & \\
A^2_1: & \bullet_1 & \bullet_1 & \bullet_1 && \\
& \bullet_{-2} \ar[r] \ar[u] & \bullet_{-1} \ar[u] \ar[r] & \bullet_0 \ar[u] \\
A^2_2: && \bullet_2 & \bullet_2  & \bullet_2 & \\
&& \bullet_1 \ar[u] & \bullet_1 \ar[u] & \bullet_1 \ar[u] & \\
& \bullet_2 & \bullet_{-2} \ar[r] \ar[u] \ar[l] \ar[dl] \ar[d] & \bullet_{-1} \ar[r] \ar[u] \ar[d] \ar[dr] & \bullet_0 \ar[r] \ar[u] & \bullet_2 \\
& \bullet_2 & \bullet_2 & \bullet_2 & \bullet_2 & .}
\end{equation}

Our goal in this paper is to prove the analogue of Theorem \ref{catalanthm} for $k$-extended action graphs.  Specifically, we would like to identify the sequences given by the number of new vertices adjoined to form each new graph in the sequences $A^k$.

\section{The main theorem} \label{main}

In this section we show that $k$-extended action graphs can be described via the $k$th convolution of the Catalan sequence.  We begin with the precise statement.

\begin{theorem}
	The number of vertices labeled by $n$ in the action graph $A^k_n$ is given by $C^k_n$.
\end{theorem}

Let us first give an outline of the idea of the proof, using the 2-extended action graphs depicted in Figure 2.  Looking at $A^2_0$, observe that if we remove the vertex labeled by -2, we simply recover $A^1_0$.  In particular, if we restrict to this part of the graph, what develops inductively are essentially the 1-extended action graphs.  What is more interesting is what grows out of the vertex $-2$.   For emphasis, let us look at the graph $A^2_1$ with the edge $\bullet_{-2} \rightarrow \bullet_{-1}$ removed:
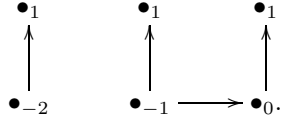
\[ \xymatrix{\bullet_1 & \bullet_1 & \bullet_1  \\
\bullet_{-2} \ar[u] & \bullet_{-1} \ar[u] \ar[r] & \bullet_0. \ar[u] } \]
We see the copy of $A^1_1$ to the right, and in subsequent steps this part of the graph develops exactly as for the 1-extended action graphs.  To the left, at this stage we have something that looks a lot like $A^0_1$, but with the initial vertex labeled by $-2$ rather than by 0.  We expect this part of the graph to develop as the original action graphs do, i.e., producing $C_n$ vertices labeled by $n$ at step $n$.  However, when we build $A^2_2$, we still have to take into account the paths that start at the vertex labeled by $-2$ and end at vertices labeled by 1 in the right-hand side.  Since the vertex labeled by $-2$ is initial in any 2-extended action graph, we produce new vertices and edges in this way at every step.  To illustrate this phenomenon, consider the following decomposition of the graph $A^2_2$, observing that we have repeated the vertex labeled by $-2$ to emphasize the three different portions of the graph that grow out of it:
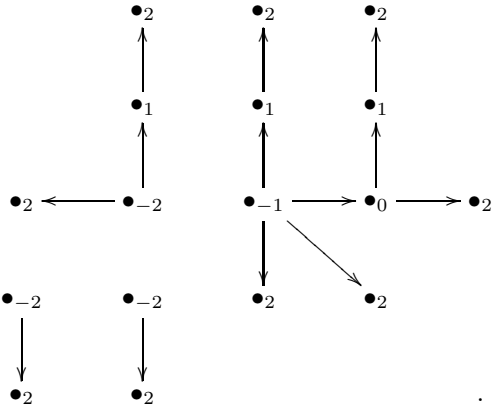
\[ \xymatrix{& \bullet_2 & \bullet_2  & \bullet_2 & \\
& \bullet_1 \ar[u] & \bullet_1 \ar[u] & \bullet_1 \ar[u] & \\
\bullet_2 & \bullet_{-2} \ar[u] \ar[l] & \bullet_{-1} \ar[r] \ar[u] \ar[d] \ar[dr] & \bullet_0 \ar[r] \ar[u] & \bullet_2 \\
\bullet_{-2} \ar[d] & \bullet_{-2} \ar[d] & \bullet_2 & \bullet_2 & \\
\bullet_2 & \bullet_2 &&& . }\]
The two single-edge graphs in the lower-left corner now develop as the ordinary actions graphs, just one stage behind.  The idea behind our proof is to iterate this phenomenon in the general case.

\begin{proof}
	Let $a^k_n$ denote the number of vertices labeled by $n$ in the $k$-extended action graph; we thus want to prove that $a^k_n = C^k_n$.  We prove this result by induction on $k$.  The $k=0$ case is precisely Theorem \ref{catalanthm}.
	
	Assuming by our inductive hypothesis that $a^{k-1}_j = C^{k-1}_j$ for all $j$, it suffices to show that 
	\[ a^k_n = \sum_{i=0}^na_i^{k-1} C_{n-i}. \]
	Observe that the right-hand side of this equation is equal to 
	\[ a_n^{k-1} + \sum_{i=0}^{n-1} a_i^{k-1}C_{n-i}, \]
	so we prove that we can obtain $a_n^k$ as such a sum.
	
	As in the $k=2$ example given above, let us divide the graph into two parts: the right-hand side, arising only from the vertices labeled by $-k+1, \ldots, 0$, and the left-hand side, which develops from the vertex labeled by $-k$.  We know by induction that on the right-hand side at step $n$ we get $a_n^{k-1}$ new vertices labeled by $n$.  So, it remains to show that the left-hand side produces new vertices enumerated by the rest of this sum.  
	
	We claim that this part of the sum is obtained from subgraphs that originate from paths starting at the vertex labeled by $-k$ and ending at vertices on the right-hand side of the graph.  To start, we produce a new edge $\bullet_{-k} \rightarrow \bullet_1$ corresponding to the path starting at the vertex labeled by $-k$ and ending at the vertex labeled by 0, which is in the right-hand side of the graph.  There is only one such path, which we can think of as being counted by $a^{k-1}_0$.  From this new edge, the graph develops as for ordinary action graphs, so at step $n$ contributes $C_n$ new vertices.  We similarly produce edges $\bullet_{-k} \rightarrow \bullet_2$ corresponding to paths starting at the vertex labeled by $-k$ and ending at vertices labeled by 1 on the right-hand side. There are $a^{k-1}_1$ such vertices, by our inductive hypothesis, and in subsequent steps these new subtrees grow as ordinary action graphs do, but they originate one step later.  Thus, at step $n$, each of these $a^{k-1}_1$ subtrees gains $C_{n-1}$ new vertices.  More generally, at step $i$ we obtain edges $\bullet_{-k} \rightarrow \bullet_i$ corresponding to paths starting at the vertex labeled by $-k$ and ending at vertices labeled by $i-1$ on the right-hand side, of which there are $a^{k-1}_i$.  At step $n$, the subgraph growing out of these vertices gains $C_{n-i}$ new vertices.    Thus, the total number of vertices labeled by $n$ obtained in this fashion is given by the indicated sum.
\end{proof}

\section{Planar rooted forests} \label{prf}

In the earlier paper \cite{alv}, the authors give a direct proof of Theorem \ref{catalanthm}, but they also show that the new vertices in each $A^0_n$ are in bijection with planar rooted trees with $n$ edges.  The latter set is one of the standard ways to obtain the $n$-th Catalan number $C_n$.  

As we would like an analogue for $k$-extended action graphs, we generalize planar rooted trees to planar rooted forests consisting of $k+1$-trees.  In this section we show that the family of such forests also gives the sequence $C^k$.  Let us begin with definitions.

\begin{definition}
A \emph{tree} is a connected graph with no loops or multiple edges.  A \emph{rooted tree} is a tree with a specified vertex called the \emph{root}.
\end{definition}

\begin{definition}
A \emph{leaf} of a rooted tree is a vertex of valence 1 which is not the root.  In the special case of a tree consisting of a single vertex with no edges, we consider this vertex to be both the root and a leaf.
\end{definition}

\begin{remark}
Applying this terminology to action graphs, note that they can be regarded as (directed) trees, and that the vertices of $A^k_{n+1}$ which are not in $A^k_n$ are precisely the leaves of $A^k_{n+1}$.
\end{remark}

Here we consider \emph{planar} rooted trees.  In particular, if the bottom vertex is the root, we regard the following two trees as different:
\[ \xymatrix{\bullet &&&&& \bullet \\
\bullet \ar@{-}[u] && \bullet & \bullet && \bullet \ar@{-}[u] \\
& \bullet \ar@{-}[ul] \ar@{-}[ur] &&& \bullet \ar@{-}[ul] \ar@{-}[ur] & .} \]

\begin{theorem} \label{prt and catalan} \cite[\S 1.5]{stanleycat}
 The set of all planar rooted trees with $n$ edges contains $C_n$ elements. 
\end{theorem}

Here, we want to consider the set of forests of rooted planar trees with $k$ trees containing $n$ total edges. We seek to establish that this set contains $C_n^{k}$ elements.

\begin{definition}
A \emph{planar rooted forest} is a set of planar rooted trees.  We assume that the trees in a planar rooted forest are given a specified ordering. 
\end{definition}

In particular, the ordering distinguishes between the forests
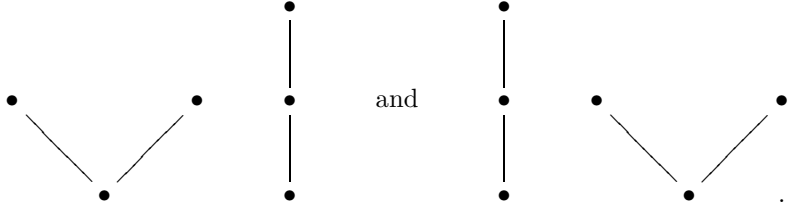
\[ \xymatrix{&&& \bullet \ar@{-}[d] && \bullet \ar@{-}[d] &&& \\
\bullet \ar@{-}[dr] && \bullet \ar@{-}[dl] & \bullet \ar@{-}[d] & \text{and} & \bullet \ar@{-}[d] & \bullet \ar@{-}[dr] && \bullet \ar@{-}[dl] \\
& \bullet && \bullet && \bullet && \bullet &.} \]

Now, we want to show that we can obtain self-convolutions of the Catalan sequence from such forests.  To do so, we use the following formula for the $n$th term of the $k$th self-convolution of the Catalan number.

\begin{prop}  \cite[\S 2]{ted} \label{formula}
The $n$th term of the sequence $C^k$ is given by the following formula:
\[  C_n^k=\sum \limits _{\substack {0\leq i_1, i_2, i_3, ..., i_k, i_{k+1}\leq n \\ \sum_{j=1}^{k+1} i_j=n}} C_{i_1} C_{i_2} C_{i_3} ... C_{i_k} C_{i_{k+1}}. \]
\end{prop}

We apply this formula to prove our first main result, which is the following.

\begin{theorem} \label{forestconv}
The set of all planar rooted forests with $k+1$ trees and $n$ edges contains exactly $C_n^k$ elements.
\end{theorem}

\begin{proof}
Consider a planar rooted forest consisting of $k+1$ trees and $n$ total edges.  For any $1 \leq j \leq k+1$, let $i_j$ denote the number of edges in the $j$th tree of this forest.  Observe that $0 \leq i_j \leq n$ and 
\[ \sum_{j=1}^{k+1} i_j = n. \]
Given any $i_j$, we know from Theorem \ref{catalanthm} that there are $C_{i_j}$ possible trees with $i_j$ edges.  Thus, given any sequence $i_1, \ldots, i_{k+1}$, there are $C_{i_1} \cdots C_{i_{k+1}}$ forests with edges distributed among the trees via the sequence.  Ranging over all such sequences, we obtain precisely the formula in Proposition \ref{formula}.
\end{proof}

\section{Comparison of planar rooted forests with action graphs} \label{comp}

In this last section, we give a more precise relationship between our two structures of interest by establishing a bijection between the number of new edges added to obtain $A^k_n$ from $A^k_{n-1}$ and the number of planar rooted forests with $k+1$ trees and $n$ edges.  

In the comparison between planar rooted trees and (0-extended) action graphs \cite{alv}, the main idea is that the trees can be assembled in such a way as to obtain the corresponding action graph.  We review the idea of this construction, so as to be able to generalize it.

We label the vertices of a planar rooted tree via the following rule.  The root vertex is always given the label 0.  Given a representative of the tree with the root at the bottom, label the vertices by successive natural numbers, moving upward from the root and from right to left.  For example, we have the labelings
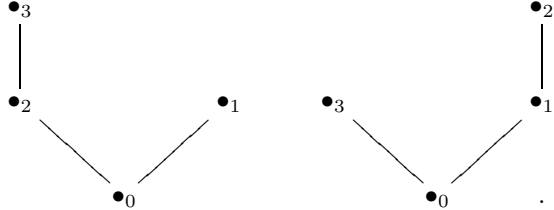
\[ \xymatrix{\bullet_3 &&&&& \bullet_2 \\
\bullet_2 \ar@{-}[u] && \bullet_1 & \bullet_3 && \bullet_1 \ar@{-}[u] \\
& \bullet_0 \ar@{-}[ul] \ar@{-}[ur] &&& \bullet_0 \ar@{-}[ul] \ar@{-}[ur] & .} \]
(Note that this convention differs slightly from the one used in previous work \cite{alv}.)  Because these labels are unique, they give a planar rooted tree a canonical directed structure, given by the usual ordering on the natural numbers.

We introduce some additional terminology.

\begin{definition}
A \emph{branch} of a rooted tree is a path from either the root or from a vertex of valence greater than 2 to a leaf.   
\end{definition}

Now we can describe the comparison with action graphs.  

We can use these labelings to assemble planar rooted trees into the appropriate action graph in the manner described in the previous paper \cite{alv}; we illustrate the $n=2$ case here.  The two planar rooted trees with two edges are given by
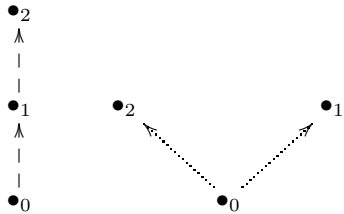
\[ \xymatrix{\bullet_2  &&& \\
\bullet_1 \ar@{-->}[u] & \bullet_2 && \bullet_1 \\
\bullet_0 \ar@{-->}[u] && \bullet_0 \ar@{..>}[ul] \ar@{..>}[ur] & } \]
and can thus be assembled to form the action graph $A_2$, as given by
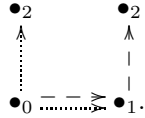
\[ \xymatrix{\bullet_2 & \bullet_2 \\
\bullet_0 \ar@{..>}[u] \ar@{-->}[r]<.5ex> \ar@{..>}[r]<-.5ex> & \bullet_1. \ar@{-->}[u]} \]

The precise statement of the comparison is as follows.

\begin{theorem} \cite[3.3]{alv}
The function assigning any planar rooted tree with $n$ edges to the leaf that it contributes to the action graph $A^0_n$ defines a bijection.
\end{theorem}

Here, we want to establish a bijection between $k$-extended action graphs and rooted planar forests with $k+1$ trees.  For convenience for this comparison, in this section we assume that a planar rooted forest with $k+1$ trees has roots labeled, from left to right, by $-k, -k+1, \ldots, -1, 0$.  Doing so faciliates the comparison with the vertices with the same labels in the $k$-extended action graphs.  From there, we follow the same labeling convention as before, working from bottom to top and right to left.   An example of such a labeling is the following:
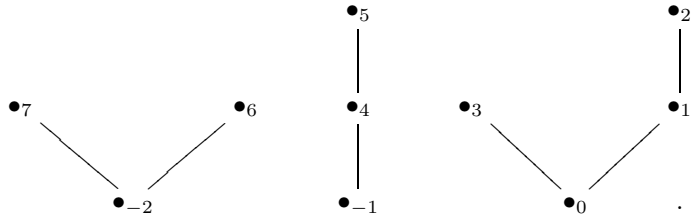
\[ \xymatrix{&&& \bullet_5 \ar@{-}[d] &  && \bullet_2 \ar@{-}[d] \\
\bullet_7 \ar@{-}[dr] && \bullet_6 \ar@{-}[dl] & \bullet_4 \ar@{-}[d] & \bullet_3 \ar@{-}[dr] && \bullet_1 \ar@{-}[dl] \\
& \bullet_{-2} && \bullet_{-1} && \bullet_0 &.} \]

We claim that we can relate these forests to more general action graphs using the same kind of assembly process as before.  To illustrate the process, let us look at the first few 2-extended action graphs.  Our initial graph
\[ \xymatrix@1{\bullet_{-2} \ar[r] & \bullet_{-1} \ar[r] & \bullet_0} \]
has one leaf, the vertex labeled by 0, which corresponds to the single planar rooted forest with 3 trees and no edges:
\[ \xymatrix@1{\bullet_{-2} & \bullet_{-1} & \bullet_0.} \]

At the next stage, observe that there are three planar rooted forests with 3 trees and 1 edge:
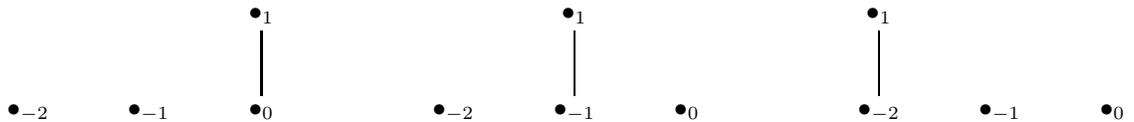
\[ \xymatrix{ && \bullet_1 &&& \bullet_1 &&& \bullet_1 && \\
\bullet_{-2} & \bullet_{-1} & \bullet_0 \ar@{-}[u] && \bullet_{-2} & \bullet_{-1} \ar@{-}[u] & \bullet_0 && \bullet_{-2} \ar@{-}[u] & \bullet_{-1} & \bullet_0} \]
which precisely assemble to give the leaves, labeled by 1, in the action graph $A^2_1$:
\[ \xymatrix{\bullet_1 & \bullet_1 & \bullet_1 & \\
\bullet_{-2} \ar[r] \ar[u] & \bullet_{-1} \ar[u] \ar[r] & \bullet_0. \ar[u]} \]

We consider one more stage, at which things get more interesting.  We have 9 planar rooted forests with 3 trees and 2 edges; the order in which we have listed them here suggests a canonical way for obtaining them from those with only one edge without repetition:
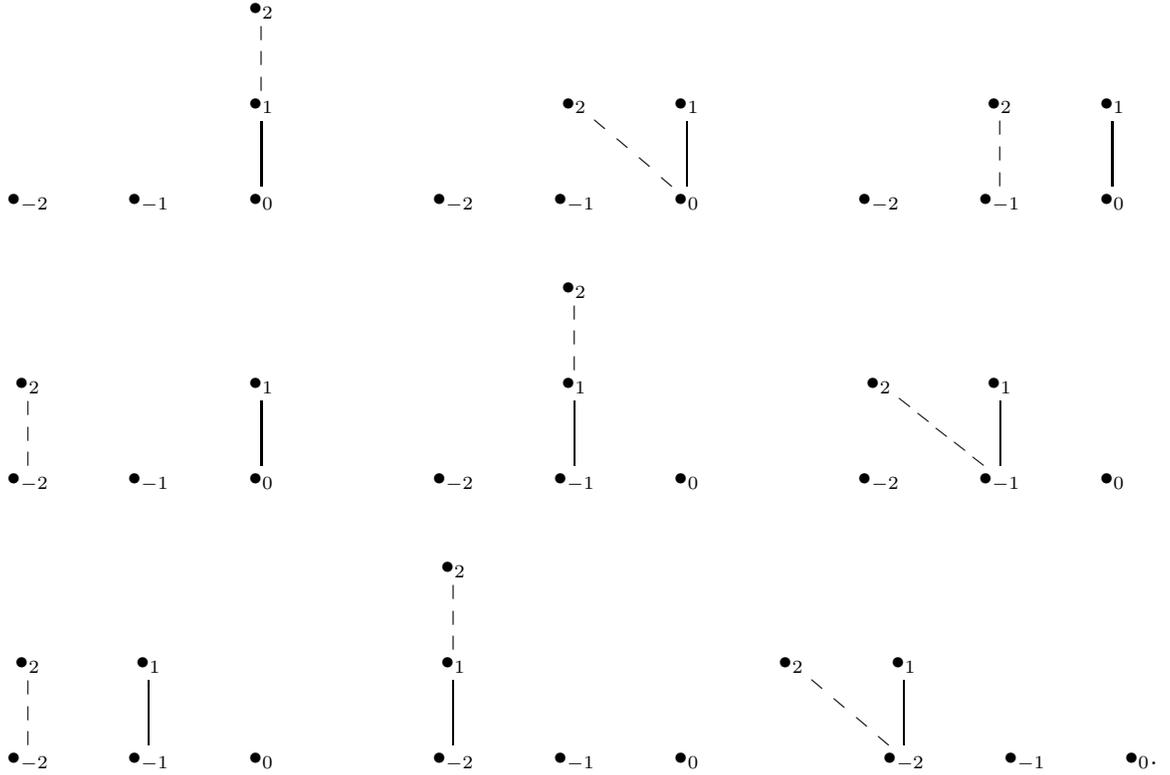
\[ \xymatrix{ && \bullet_2 \ar@{--}[d] &&&&&&&& \\
&& \bullet_1 \ar@{-}[d] &&& \bullet_2 \ar@{--}[dr] & \bullet_1 \ar@{-}[d] &&& \bullet_2\ar@{--}[d] & \bullet_1 \ar@{-}[d] \\
\bullet_{-2} & \bullet_{-1} & \bullet_0 && \bullet_{-2} & \bullet_{-1} & \bullet_0 && \bullet_{-2} & \bullet_{-1} & \bullet_0} \]

\[ \xymatrix{ &&&&& \bullet_2 \ar@{--}[d] &&&&& \\
\bullet_2 \ar@{--}[d] && \bullet_1 \ar@{-}[d] &&& \bullet_1 \ar@{-}[d] & && \bullet_2\ar@{--}[dr] & \bullet_1 \ar@{-}[d] &  \\
\bullet_{-2} & \bullet_{-1} & \bullet_0 && \bullet_{-2} & \bullet_{-1} & \bullet_0 && \bullet_{-2} & \bullet_{-1} & \bullet_0} \]

\[ \xymatrix{ &&&&  \bullet_2 \ar@{--}[d] &&&&&& \\
\bullet_2 \ar@{--}[d] & \bullet_1 \ar@{-}[d] &&& \bullet_1 \ar@{-}[d] &&&  \bullet_2\ar@{--}[dr] & \bullet_1 \ar@{-}[d] &&  \\
\bullet_{-2} & \bullet_{-1} & \bullet_0 && \bullet_{-2} & \bullet_{-1} & \bullet_0 && \bullet_{-2} & \bullet_{-1} & \bullet_0.} \]
We invite the reader to check that these dotted edges assemble to the dotted edges in $A^2_2$:
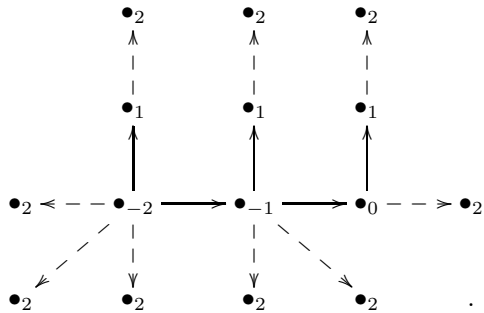
\[ \xymatrix{& \bullet_2 & \bullet_2  & \bullet_2 & \\
& \bullet_1 \ar@{-->}[u] & \bullet_1 \ar@{-->}[u] & \bullet_1 \ar@{-->}[u] & \\
\bullet_2 & \bullet_{-2} \ar[r] \ar[u] \ar@{-->}[l] \ar@{-->}[dl] \ar@{-->}[d] & \bullet_{-1} \ar[r] \ar[u] \ar@{-->}[d] \ar@{-->}[dr] & \bullet_0 \ar@{-->}[r] \ar[u] & \bullet_2 \\
\bullet_2 & \bullet_2 & \bullet_2 & \bullet_2 & .}\]
We now turn to the proof that this process always works.

\begin{theorem} \label{forestaction}
	The above procedure defines a bijection between the set of leaves in $A^k_n$ and the set of rooted planar forests with $k+1$ trees and $n$ total edges.
\end{theorem}

\begin{proof}
We assume $k$ is fixed and use induction on $n$.  When $n = 0$, we have the graph $A^k_0$, which is just
\[ \xymatrix@1{\bullet_{-k}\ar[r]  & \bullet_{-k+1} \ar[r] & \cdots \ar[r] & \bullet_1 \ar[r] & \bullet_0.} \]  
This graph has one leaf, at the vertex labeled by 0.  Correspondingly, there is exactly one planar rooted forest with $k+1$ trees and no edges, namely,
\[ \xymatrix@1{\bullet_{-k} & \bullet_{-k+1} & \cdots & \bullet_1 & \bullet_0.} \]  
Thus, we have the desired bijection when $n=0$.

Suppose the bijection holds for $n \geq 0$.   Consider the set of planar rooted forests with $n + 1$ edges and $k+1$ trees and the action graph $A^k_{n+1}$.  Assume any such forest is given its canonical labeling, as described above.  By the inductive hypothesis, the vertices of the forest labeled by $-k, \ldots, 0, 1, \ldots, n$ correspond to edges of the graph $A^k_n$.   Consider the vertex labeled by $n+1$ and the unique branch from the appropriate root vertex to it.  All but the last edge of this branch can be identified with a path in $A^k_n$; adjoin a new edge with final vertex labeled by $n+1$ to complete a path corresponding to the whole branch. 

To complete the proof, then, it suffices to show that for $n \geq 0$ the number of planar rooted forests with $k+1$ trees and $n+1$ edges corresponds exactly to the number of directed paths in $A^k_n$ ending at vertices labeled by $n$.  As illustrated by the examples above, we put the following ordering on the set of planar rooted forests with $n$ edges.  When $n=0$, order the $k+1$ forests by the root at which the edge is placed, starting with the one labeled by 0 and moving left.  Observe that these trees correspond exactly to directed paths to the vertex labeled by 0 in $A^k_0$.

For $n\geq 1$, start with the first forest in the inductively-defined ordering on planar rooted forests with $n$ edges.  Obtain a forest with $n+1$ edges by adjoining a new edge to the vertex labeled by $n$. Moving top to bottom and left to right, obtain new forests by adjoining the extra edge to each vertex.  Observe that these vertices are are in bijection with directed paths in $A^k_n$ ending at the appropriate vertex labeled by $n$.  Repeat for the other planar rooted forests with $n$ edges, in order.  Since this ordering was chosen to be compatible with the vertex labeling convention, we have assured that there is no repetition of graphs.  Thus we see that planar rooted forests $k+1$ trees and $n+1$ edges are in bijection with paths in $A^k_n$ ending at a vertex labeled by $n$, as we wished to show.
\end{proof}

\end{document}